\documentclass[11pt]{amsart}
\usepackage{amssymb}
\usepackage{eurosym}
\usepackage{amsfonts}
\usepackage{amsmath}
\usepackage{enumerate}
\usepackage{color}
\usepackage{cite}
\usepackage{tikz}

\usetikzlibrary{patterns}
\newtheorem{theorem}{Theorem}[section]

\newtheorem{corollary}[theorem]{Corollary}

\newtheorem{example}[theorem]{Example}

\newtheorem{proposition}[theorem]{Proposition}
\newtheorem{remark}[theorem]{Remark}

\numberwithin{equation}{section} \textwidth=15.5cm
\begin{document}
\title[Singular $\phi$-Laplacian]{Potential systems with singular $\phi$-Laplacian}
\author{Petru Jebelean}
\address{Institute for Advanced Environmental Research, West University of Timi\c{s}oara,
Blvd. V. P\^{a}rvan, no. 4, 300223  Timi\c{s}oara, Romania,
E-mail: petru.jebelean@e-uvt.ro}

\begin{abstract} We are concerned with solvability of the boundary value problem
$$-\left[ \phi(u^{\prime}) \right] ^{\prime}=\nabla_u F(t,u), \quad
\left ( \phi \left( u^{\prime }\right)(0), -\phi \left( u^{\prime }\right)(T)\right )\in \partial j(u(0), u(T)),$$
where $\phi$ is a homeomorphism from $B_a$ -- the open ball of radius $a$ centered at $0_{\mathbb{R}^N},$ onto  $\mathbb{R}^N$, satisfying $\phi(0_{\mathbb{R}^N})=0_{\mathbb{R}^N}$, $\phi =\nabla \Phi$, with $\Phi: \overline{B}_a \to (-\infty, 0]$ of class $C^1$ on $B_a$, continuous and strictly convex on $\overline{B}_a.$ The potential $F:[0,T] \times \mathbb{R}^N \to \mathbb{R}$ is of class $C^1$ with respect to the second variable and $j:\mathbb{R}^N \times \mathbb{R}^N \rightarrow (-\infty, +\infty]$ is proper,  convex and lower semicontinuous. We first provide a variational formulation in the frame of critical point theory for convex, lower semicontinuous perturbations of $C^1$-functionals. Then, taking the advantage of this key step, we obtain existence of minimum energy as well as saddle-point solutions of the problem. Some concrete illustrative examples of applications are provided.

\end{abstract}

\maketitle

\noindent Mathematics Subject Classification: 34B15, 34L30, 47J30, 47N20

\noindent Keywords and phrases: Singular $\phi$-Laplacian; Maximal monotone operator; A priori estimate; Subdifferential; Critical point; Palais-Smale condition; Saddle-point theorem
\bigskip
\section{Introduction}

\noindent In this paper we deal with solvability of a system having the form
\begin{equation}
-\left[ \phi(u^{\prime}) \right] ^{\prime}=\nabla_u F(t,u) \quad \mbox{in }[0,T],  \label{sysp}
\end{equation}
associated with the potential multivalued boundary condition
\begin{equation}\label{bcsysp}
 \left ( \phi \left( u^{\prime }\right)(0), -\phi \left( u^{\prime }\right)(T)\right )\in \partial j(u(0), u(T)),
\end{equation}
under the following hypotheses:
$$
\begin{array}{clrc}
(H_{\Phi})  & \phi \, \, \mbox{\sl is a homeomorphism from } B_a \, \, \mbox{\sl onto } \mathbb{R}^N, \, \, \mbox{\sl such that }\phi(0_{\mathbb{R}^N})=0_{\mathbb{R}^N}, \, \phi =\nabla \Phi, \, \, \mbox{\sl with }\\
&\Phi: \overline{B}_a \to (-\infty, 0] \, \, \mbox{\sl of class } C^1 \, \mbox{\sl on } B_a, \, \, \mbox{\sl continuous and strictly convex on } \overline{B}_a;
\end{array}
$$
$$
\begin{array}{clrc}
(H_{F})  & F \! = \! F(t,u) \!  : \!  [0,T] \times \mathbb{R}^N \to \mathbb{R} \, \, \mbox{\sl is continuous,  } \nabla_u F \, \, \mbox{\sl exists and  is continuous on the set}   \\
& [0,T] \times \mathbb{R}^N \mbox{\sl and } F(\, \cdot \, , 0_{\mathbb{R}^N})=0;
\end{array}
$$
$$
\begin{array}{clrc}
(H_{j})  & j:\mathbb{R}^N \times \mathbb{R}^N \rightarrow (-\infty, +\infty] \, \, \mbox{\sl is proper,  convex, lower semicontinuous,  }j(0_{\mathbb{R}^N \times \mathbb{R}^N})=0 \\
& \mbox {\sl and }  0_{\mathbb{R}^N \times \mathbb{R}^N}\in \partial j (0_{\mathbb{R}^N \times \mathbb{R}^N}).
\end{array}
$$

\smallskip
\noindent Here, for $0<r<\infty$, $B_r$ denotes the open ball of radius $r$ centered at the origin in $\mathbb{R}^N$ -- considered with the Euclidean norm $| \cdot |,$ and  $\partial j$ stands for the subdifferential of $j$ in the sense of convex analysis \cite{Rock}. As $a$ in ($H_{\Phi}$)  is finite, according to an already usual terminology, the $\phi$-Laplacian operator $u \mapsto \left[ \phi(u^{\prime}) \right] ^{\prime}$ is called "singular".
It is worth to notice that assumption of strict convexity of $\Phi$ in hypothesis $(H_{\Phi})$ ensures that $\phi$ is strictly monotone on $B_a$ while, by the lower semicontinuity assumption on the proper convex function $j$ in ($H_{j}$),  the multivalued operator $\partial j : \mathbb{R}^N \times \mathbb{R}^N \rightarrow 2^{\mathbb{R}^N \times \mathbb{R}^N}$ is maximal monotone \cite{Br, Rock}.

\medskip
The prototype of singular $\phi$-Laplacians is the relativistic operator
 \begin{equation*}
\displaystyle u \mapsto {\mathcal R}u:= \left[ \frac{u^{\prime}}{\sqrt{1-|u^{\prime}|^2}} \right] ^{\prime},
\end{equation*}
which is engendered by
$$\phi(y)=\frac{y}{\sqrt{1- |y|^2}}\quad (y\in B_1)$$
and has the potential $\Phi:\overline{B}_1 \to (-\infty, 0]$,  $\Phi(y)=-\sqrt{1- |y|^2}$. The operator ${\mathcal R}$ occurs in the dynamics of special relativity \cite{berg, LaLi}. So, for $N=3$, it is involved in
equations describing the motion of some mechanical systems when the Newtonian acceleration is replaced by
the relativistic one (with the velocity of light normalized to $1$). Also, a natural extension of ${\mathcal R}$ appears as being the so-called $p$-relativistic operator
\begin{equation*}
\displaystyle u \mapsto {\mathcal R}_pu:= \left[ \frac{|u^{\prime}|^{p-2}u^{\prime}}{(1-|u^{\prime}|^p)^{1-1/p} }\right] ^{\prime} \qquad (p>1)
\end{equation*}
which is another relevant example of singular $\phi$-Laplacian having the structure required by hypothesis $(H_{\Phi})$  \cite{JeMaSe, JeSe1}.

\medskip
By \textit{a solution} of the boundary value problem \eqref{sysp}-\eqref{bcsysp} we mean a function $u\in C^{1}:=C^{1}( \left[ 0,T \right]; \mathbb{R}^N)$ with  $\left\vert u^{\prime}(t)\right\vert <a$
for all $t\in  [0,T]$, such that $\phi(u^{\prime })\in C^{1}$, $(u(0), u(T))\in D\left ( \partial j \right )$ and which satisfies \eqref{sysp} and \eqref{bcsysp}.

\smallskip
Due to the boundedness of the derivative of a solution, it is noteworthy that its values at the end points $0$ and $T$ cannot be distant anyway.  Precisely, if for $0<\sigma < \infty$, we define the diagonal strip-like set
$$D_{\sigma }:=\{ (x,y)\in \mathbb{R}^N \times \mathbb{R}^N \, : \, |x-y|< \sigma  \},$$
and $u$ is a solution of problem \eqref{sysp}-\eqref{bcsysp} then, since
\begin{equation}\label{dirstrip}
\left | u(T)-u(0) \right | \leq \int_0^T\left |u'(t) \right | dt  <T a,
\end{equation}
clearly one has that
\begin{equation}\label{stripsol}
\left ( u(0),u(T) \right )\in D_{Ta} \cap D\left ( \partial j \right ).
\end{equation}

In recent years, much attention has been paid to the study of boundary value problems with singular $\phi$-Laplacian. Mainly, the obtained results concern the solvability and multiplicity of solutions of nonlinear systems or scalar equations subject to several usual boundary conditions, such as Dirichlet, Neumann or periodic,  under various assumptions on the perturbing nonlinearity  (see e.g., \cite{Am-Ar, Ar-Be-To, Ber-Maw_jde, Ber-Maw0, BoDF-SIAM, BoDP, JePr2}). Besides the topological techniques -- such as Leray-Schauder degree, fixed point index or lower and upper solutions, the variational approach proved to be very effective in treatment of this class of problems. Thus, starting from a formulation in terms of variational inequality due to H. Brezis and J. Mawhin \cite{BrMa0} (also see \cite{BeJeMaCV}) in paper \cite{BeJeMaRL} it is shown for the first time that such boundary value problems lend themselves very well to the critical point theory for convex, lower semicontinuous perturbations of $C^1$-functionals  developed by A. Szulkin \cite{Sz}. Further adaptations  of this theory or Hamiltonian formulations, closely related to the special case of the relativistic operator ${\mathcal R},$ enabled the treatment of various systems,  among others concerning Lorentz force equation  \cite{Ar-Be-To, BoDP} or the relativistic Kepler problem \cite{BoDF-SIAM}.

\medskip
On the other hand, the boundary condition \eqref{bcsysp} is a general one, which primarily covers the classical ones. Thus, for instance, denoting by $I_K$ the indicator function of a
nonempty, closed and convex set $K\subset \mathbb{R}^N \times \mathbb{R}^N$, the Dirichlet and Neumann homogeneous boundary
conditions
\begin{equation}\label{dirh}
 \quad u(0)= 0_{\mathbb{R}^N}=u(T) ,
\end{equation}
\begin{equation}\label{neuh}
\phi(u')(0)= 0_{\mathbb{R}^N}=\phi(u')(T)
\end{equation}
are obtained by choosing $j=I_K$ with $K=\{ 0_{\mathbb{R}^N \times \mathbb{R}^N} \}$ and $K=\mathbb{R}^N \times \mathbb{R}^N$, respectively. Also, setting
$$d_N^1:=\left \{ (x,x) \, : \, x\in \mathbb{R}^N\right \}  \quad \mbox{ and } \quad d^2_N:=\left \{ (x,-x) \, : \, x\in \mathbb{R}^N\right \}, $$
the periodic and antiperiodic boundary conditions
\begin{equation}\label{peri}
\quad u(0)-u(T)=0_{\mathbb{R}^N}=\phi(u')(0)-\phi(u')(T)  ,
\end{equation}
\begin{equation}\label{aperi}
u(0)+u(T)= 0_{\mathbb{R}^N}=\phi(u')(0)+\phi(u')(T)
\end{equation}
are obtained with $K=d^1_N$ and  $K=d^2_N$, respectively. More general, following \cite[Section 3.5]{PJ}, let $g:\mathbb{R}^N \times \mathbb{R}^N \to \mathbb{R}$ be a convex and G\^{a}teaux differentiable
function with $g(0_{\mathbb{R}^N \times \mathbb{R}^N})=0$ and $dg(0_{\mathbb{R}^N \times \mathbb{R}^N})=0_{\mathbb{R}^N \times \mathbb{R}^N}$, where $dg$ stands for the differential of $g$. Given a closed and convex set $K\subset \mathbb{R}^N \times \mathbb{R}^N$ with $0_{\mathbb{R}^N \times \mathbb{R}^N} \in K$, let $N_K(z)$ be the normal cone to $K$ at $z\in K$, that is
$$N_K(z)=\left \{\xi \in \mathbb{R}^N \times \mathbb{R}^N \, : \,  \langle \!  \langle  \xi | x-z \rangle \! \rangle \leq 0 \mbox{ for all }x\in K  \right \},$$
where $\langle \!  \langle  \cdot | \cdot \rangle \!  \rangle$ is the usual inner product on the product space $\mathbb{R}^N \times \mathbb{R}^N$. It is well known that $\partial I_K(z)=N_K(z)$ if $z\in K$ and  $\partial I_K(z)=\emptyset$ if $z\notin K$ (see e.g., p. 215-216 in \cite{Rock}). Then, taking $j=g+I_K$, we have that $j$ satisfies $(H_{j})$ and the boundary condition \eqref{bcsysp} reads
\begin{equation}\label{ncon}
\left ( u(0), u(T) \right ) \in K, \qquad \left ( \phi \left( u^{\prime }\right)(0), -\phi \left( u^{\prime }\right)(T)\right ) -dg (u(0),u(T))\in N_K(u(0), u(T)).
\end{equation}

For various possible choices of $g$ and $K$ in \eqref{ncon}, we refer the reader to e.g. \cite{Beh, HaPa, Je0 ,PJ, JeMo1, KaPamw, KoPaPa}. As it can be seen therein, problems with a boundary condition of the type \eqref{bcsysp} have been studied for operators such as the classical Laplacian $u\mapsto u''$, the vector $p$-Laplacian $ u\mapsto \left ( |u'|^{p-2}u' \right )'$ or other monotone differential operators. We also note that earlier works deal
with differential equations with boundary conditions of type \eqref{bcsysp}. In this respect, let us remark that Section 5.2 in \cite{Ba} is devoted to the study of a second-order multivalued equation in a Hilbert space submitted to a two-point boundary condition of type \eqref{bcsysp}, while in \cite{MoPe} higher order scalar differential equations are considered with boundary conditions in terms of a nonlinear maximal monotone mapping. In this view, an important aspect of novelty in our study is that the differential operator engendered by $\phi$  is a singular one. As will be seen, this will induce specific difficulties as well as features that do not occur with the other aforementioned operators. Such a simple aspect is relation \eqref{stripsol} which is satisfied by any solution $u$ of problem \eqref{sysp}-\eqref{bcsysp} -- we refer the reader here, for example, to Remark \ref{remark22} ($ii$) below for a detailing on one of its impact in this regard.

\bigskip
The rest of the paper is organized as follows. In Section \ref{sectiunea2} we consider an auxiliary problem in which the system is simpler than \eqref{sysp} but it is subject to a more general boundary condition -- in the sense that $\partial j$ is replaced by a maximal monotone operator.  Using arguments from the theory of maximal monotone operators, as well as from topological fixed point theory, we show that the auxiliary problem has an unique solution (Theorem \ref{rpmax}). In Section \ref{sectiunea3} we provide a variational formulation of the problem \eqref{sysp}-\eqref{bcsysp} in the frame of the critical point theory for convex, lower semicontinuous perturbations of $C^1$-functionals. Here the existence and uniqueness result from the previous section plays a key role in proving that any critical point of the energy functional is actually a solution of problem \eqref{sysp}-\eqref{bcsysp} (Theorem \ref{critpoint}). Section \ref{sectiunea4} is devoted to solvability of problem \eqref{sysp}-\eqref{bcsysp}. Taking the advantage of the variational formulation from Section \ref{sectiunea3}, we obtain the existence of minimum energy solutions in two main situations. First we show that when an eigenvalue-like constant depending on $D(j)$ is positive, the problem is solvable, without any additional conditions on $F$ (Theorem \ref{l1poz}) - this is a so-called "universal" existence result. Also, we prove the existence of minimum energy solutions when the minimization of the energy functional can be reduced over the set of all functions having a uniformly bounded average (Theorem \ref{minlem}). The corollaries of this result concern nonlinearities $F$ which are periodic  (Corollary \ref{Fperiodic}) or have anti-coerciveness properties (Corollary \ref{talp1}). The complementary cases of coerciveness are also addressed and in these situations saddle-point solutions are obtained (Theorem \ref{talp2}, Theorem \ref{talpcoerc}). Several examples of applications illustrating the general results are provided.

\section{An auxiliary existence and uniqueness result}\label{sectiunea2}

\noindent
On $\mathbb{R}^N$ we consider the usual inner product $\langle \cdot | \cdot \rangle$ engendering the norm $| \cdot |$. The notation $\| \cdot \|$ stands for the norm on $\mathbb{R}^N \times \mathbb{R}^N$ corresponding to the
inner product $\langle \!  \langle  \cdot | \cdot \rangle \!  \rangle$ introduced in the previous section.
 By $\| \cdot \| _{\infty}$ we mean the uniform norm on $C:=C([0,T]; \mathbb{R}^N)$
and the usual norm on $L^p:=L^p([0,T]; \mathbb{R}^N)$ will be denoted by $\| \cdot \|_{L^p}$ ($1\leq p \leq \infty$). The Sobolev space $W^{1, \infty}:=W^{1, \infty}([0,T]; \mathbb{R}^N)$  is considered as being endowed with the norm $\|u\|_{W^{1, \infty}}=\|u\|_{L^{\infty}}+\|u^{\prime}\|_{L^{\infty}}.$

\medskip
Let $h\in C$ be fixed and $\gamma: \mathbb{R}^N \times \mathbb{R}^N \to 2^{\mathbb{R}^N \times \mathbb{R}^N}$ be a monotone operator. We consider the problem
$$
-\left[ \phi(u^{\prime}) \right] ^{\prime}+u = h(t),\qquad \left ( \phi \left( u^{\prime }\right)(0), -\phi \left( u^{\prime }\right)(T)\right )\in \gamma(u(0), u(T)).  \leqno{
%({\mathcal P}_\gamma)\equiv
[{\mathcal P}_\gamma (h)]
} $$

The notion of \textit{solution} for the problem $[{\mathcal P}_\gamma (h)]$ is perfectly similar to that of the solution for \eqref{sysp}-\eqref{bcsysp} from the previous section, with the difference that instead of $\partial j $ is $\gamma$. Notice that if $u$ is a solution of $[{\mathcal P}_\gamma (h)]$ then necessarily $(u(0), u(T))\in D_{Ta} \cap D\left ( \gamma \right ).$

\begin{proposition}\label{p1} Problem $[{\mathcal P}_\gamma (h)]$ has at most one solution.
\end{proposition}

\begin{proof}
If $u_1$ and $u_2$ are solutions of $[{\mathcal P}_\gamma (h)]$,  from the monotonicity of $\gamma$ it follows
$$\omega (u_1,u_2) :=\langle \phi(u_1^{\prime})(0) - \phi(u_2^{\prime})(0)\, | \, u_1(0)-u_2(0) \rangle
- \langle \phi(u_1^{\prime})(T) - \phi(u_2^{\prime})(T)\, | \,   u_1(T)-u_2(T) \rangle $$
$$= \! \langle \! \langle \left ( \phi \left( u_1^{\prime }\right)(0), -\phi \left( u_1^{\prime }\right)(T)\right ), \left ( \phi \left( u_2^{\prime }\right)(0), -\phi \left( u_2^{\prime }\right)(T)\right )  |
\left ( u_1(0), u_1(T) \right ) -\left ( u_2(0), u_2(T) \right )\rangle \! \rangle \geq 0.
$$

\medskip
\noindent Then, multiplying the equality
$$-\left [ \phi(u_1^{\prime}) - \phi(u_2^{\prime})\right ] ^{\prime}+u_1-u_2=0$$
by $(u_1-u_2)$ and integrating by parts over $[0,T]$, we get
\begin{equation*}
\omega(u_1, u_2)+ \int_0^T\langle  \phi(u_1^{\prime}) - \phi(u_2^{\prime})\, | \,  u_1^{\prime}-u_2^{\prime} \rangle + \|u_1-u_2\|_{L^2}^2=0,
\end{equation*}
which implies $u_1=u_2$, by the monotonicity of $\phi$.
\end{proof}

\medskip
\begin{remark}\label{remark1}{\em From the previous proof, we observe that  the conclusion of Proposition \ref{p1} still remains valid if the entire hypothesis $(H_{\Phi})$ is replaced by the weaker one: "\textit{$\phi:B_a \to \mathbb{R}^N$ is a monotone mapping}".}
\end{remark}

Next, for arbitrary $x,y\in \mathbb{R}^N$, we need to make use of the Dirichlet and Neumann problems

$$-\left[ \phi(u^{\prime}) \right] ^{\prime}+u = h(t),\qquad u(0)=x, \,  u(T)=y,  \leqno{({\mathcal D}_{x,y})}$$
respectively,
$$-\left[ \phi(u^{\prime}) \right] ^{\prime}+u = h(t),\qquad \phi(u^{\prime})(0)=x, \,  \phi(u^{\prime})(T)=y.  \leqno{({\mathcal N}_{x,y})} $$

\begin{proposition}\label{teu} (i) Problem (${\mathcal D}_{x,y}$) is solvable if and only if
\begin{equation}\label{condex}
|y-x|<  T a.
\end{equation}
If a solution of (${\mathcal D}_{x,y}$) exists, then it is unique.

\medskip
(ii) Problem (${\mathcal N}_{x,y}$) has an unique solution for any $x,y\in \mathbb{R}^N$.
\end{proposition}

\begin{proof} First, observe that $({\mathcal D}_{x,y})$ and (${\mathcal N}_{x,y}$) are problems of type
$[{\mathcal P}_\gamma (h)]$. Indeed, for (${\mathcal D}_{x,y}$) one take $D(\gamma)=\{ (x,y)\}$, $\gamma (x,y)= \mathbb{R}^N \times \mathbb{R}^N$, while
(${\mathcal N}_{x,y}$) is obtained with $D(\gamma)=\mathbb{R}^N \times \mathbb{R}^N$, $\gamma (\xi,\eta)= \{(x,-y)\}$ for all $(\xi, \eta) \in \mathbb{R}^N \times \mathbb{R}^N$.
Thus, the uniqueness part of both $(i)$ and $(ii)$ follows by Proposition \ref{p1}.

\medskip
If \eqref{condex} holds, then (${\mathcal D}_{x,y}$) has a solution by \cite[Corollary 6]{Ber-Maw0}. Reciprocally, if (${\mathcal D}_{x,y}$) has a solution $u$, then $|y-x|<Ta$ follows from \eqref{dirstrip}.
The fact that (${\mathcal N}_{x,y}$) is solvable is a particular case of \cite[Example 1]{Ber-Maw0}.
\end{proof}

\medskip
\begin{remark}\label{remark2}
{\em In fact, as can be seen from \cite[Example 1]{Ber-Maw0} and Remark \ref{remark1}, statement \textit {(ii)} in Proposition \ref{teu} still remains true if we raplace $(H_{\Phi})$ by the weaker hypothesis: "\textit{$\phi$ is a monotone homeomorphism from $B_a$ onto $\mathbb{R}^N$ }". Not the same situation is for the statement \textit{(i)}, because in this case, the entire hypothesis $(H_{\Phi})$ is needed for the construction of the fixed point operator involved in the proof of \cite[Corollary 6]{Ber-Maw0}.

\smallskip
A different situation occurs in the particular case $N=1$, when instead of \cite[Example 1]{Ber-Maw0} and \cite[Corollary 6]{Ber-Maw0}, we can invoke \cite[Corollary 3]{Ber-Maw}, respectively \cite[Corollary 1]{Ber-Maw}. Thus, in this case, we obtain that both of the conclusions \textit{(i)} and \textit{(ii)} of Proposition \ref{teu} hold true with hypothesis "\textit{$\phi:(-a,a) \to \mathbb{R}$ is an increasing homeomorphism}" instead of $(H_{\Phi})$.}
\end{remark}

\medskip
\noindent {\bf Notations}. In the sequel, we denote by $\overline{u}_{x,y}$ the solution of (${\mathcal N}_{x,y}$) and by  $u_{x,y}$ the solution of (${\mathcal D}_{x,y}$) -- if \eqref{condex} holds true.

\begin{proposition}\label{contubar} The mapping $\mathbb{R}^N \times \mathbb{R}^N \ni (x,y)  \mapsto (\overline{u}_{x,y}(0), \overline{u}_{x,y}(T)) \in \mathbb{R}^N \times \mathbb{R}^N$ is continuous.
\end{proposition}

\begin{proof} Let $(x,y)\in \mathbb{R}^N \times \mathbb{R}^N$. Integrating
\begin{equation}\label{soneu}
-\left[ \phi(\overline{u}_{x,y}^{\prime}) \right]^{\prime}+\overline{u}_{x,y}=h(t)
\end{equation}
over $[0,T]$ we obtain
\begin{equation}\label{medneu}
\int_0^T\overline{u}_{x,y}(t)dt =y-x+\int_0^Th .
\end{equation}
Also, from
$$\phi(\overline{u}_{x,y}^{\prime})(s)=x+\int_0^s\left (\overline{u}_{x,y}-h  \right )d \tau \qquad (s\in [0,T])$$
we deduce
\begin{equation*}
\overline{u}_{x,y}(t)=\overline{u}_{x,y}(0)+\int_0^t \phi ^{-1} \left (x +\int_0^s ( \overline{u}_{x,y}-h)d \tau \right ) ds \qquad (t\in [0,T])
\end{equation*}
and, combining this with \eqref{medneu}, it follows
\begin{equation}\label{ubar0}
\overline{u}_{x,y}(0)=\frac{1}{T} \left [y-x+ \int_0^Th  -\int_0^T \left [ \int_0^t \phi ^{-1} \left (x +\int_0^s ( \overline{u}_{x,y}-h)d \tau \right ) ds \right ] dt  \right ]
\end{equation}
Similarly, we derive
\begin{equation}\label{ubarT}
\overline{u}_{x,y}(T)=\frac{1}{T} \left [y-x+ \int_0^Th  -\int_0^T \left [ \int_t^T \phi ^{-1} \left (y +\int_T^s ( \overline{u}_{x,y}-h)d \tau \right ) ds \right ] dt  \right ]
\end{equation}
 Now, let $(x_n,y_n) \rightarrow (x,y)$ in $\mathbb{R}^N \times \mathbb{R}^N$, as $n \to \infty$. From \eqref{ubar0} and \eqref{ubarT} and the boundedness of the range of $\phi^{-1}$ we infer that the sequences $\{\overline{u}_{x_n,y_n}(0)\}$
and $\{\overline{u}_{x_n,y_n}(T)\}$ are bounded in $\mathbb{R}^N$. Using the monotonicity of $\phi$ and integration by parts formula, we have the estimate
$$0=-\int_0^T \left \langle \left[ \phi(\overline{u}_{x_n,y_n}^{\prime}) -\phi( \overline{u}_{x,y}^{\prime})\right]^{\prime} \, | \, \overline{u}_{x_n,y_n} - \overline{u}_{x,y} \right \rangle + \| \overline{u}_{x_n,y_n} - \overline{u}_{x,y}\|_{L^2}^2$$
$$\geq \left \langle x_n-x \, | \, \overline{u}_{x_n,y_n}(0)-\overline{u}_{x,y}(0) \right \rangle - \left \langle y_n-y \, | \, \overline{u}_{x_n,y_n}(T)-\overline{u}_{x,y}(T) \right \rangle + \| \overline{u}_{x_n,y_n} - \overline{u}_{x,y}\|_{L^2}^2$$
which gives
$$\| \overline{u}_{x_n,y_n} - \overline{u}_{x,y}\|_{L^2}^2 \leq \left \langle y_n-y \, | \, \overline{u}_{x_n,y_n}(T)-\overline{u}_{x,y}(T) \right \rangle-\left \langle x_n-x \, | \, \overline{u}_{x_n,y_n}(0)-\overline{u}_{x,y}(0) \right \rangle,$$
implying that $\overline{u}_{x_n,y_n} \rightarrow \overline{u}_{x,y}$ in $L^2$ and hence in $L^1$, as $n \to \infty$.

\smallskip
To complete the proof, we show that
\begin{equation}\label{primconv}
\overline{u}_{x_n,y_n}(0) \rightarrow \overline{u}_{x,y}(0), \quad \mbox{as } n \to \infty;
\end{equation}
similar argument for $\overline{u}_{x_n,y_n}(T) \rightarrow \overline{u}_{x,y}(T)$. In this respect, since for any $s\in [0,T]$,
$$ \left | x_n +\int_0^s ( \overline{u}_{x_n,y_n}-h)d \tau - x -\int_0^s ( \overline{u}_{x,y}-h)d \tau\right | \leq |x_n-x| +\|\overline{u}_{x_n,y_n}-\overline{u}_{x,y}\|_{L^1} \to 0,$$
we have
\begin{equation*}
\phi^{-1}\left( x_n +\int_0^s ( \overline{u}_{x_n,y_n}-h)d \tau \right) \to \phi^{-1}\left( x +\int_0^s ( \overline{u}_{x,y}-h)d \tau \right), \quad \mbox{as }n\to \infty,
\end{equation*}
uniformly with $s\in[0,T].$ It follows
\begin{equation*}
\beta_n:=\int_0^T\left | \phi^{-1}\left( x_n +\int_0^s ( \overline{u}_{x_n,y_n}-h)d \tau \right)- \phi^{-1}\left( x +\int_0^s ( \overline{u}_{x,y}-h)d \tau \right) \right | ds \to 0, \quad \mbox{as }n \to \infty.
\end{equation*}
Then, using \eqref{ubar0}, it is straightforward to see that
\begin{equation*}
\left | \overline{u}_{x_n,y_n}(0) - \overline{u}_{x,y}(0)\right | \leq \frac{1}{T} \left [ |y_n-y|+|x_n-x|+ T\beta _n \right ],
\end{equation*}
which clearly implies \eqref{primconv}.
\end{proof}

\medskip
\noindent Now, define $\theta : \mathbb{R}^N \times \mathbb{R}^N \to \mathbb{R}^N \times \mathbb{R}^N$ by setting

$$D(\theta):=D_{T a}, \qquad \theta (x,y):= \left( -\phi(u_{x,y}^{\prime})(0),\phi(u_{x,y}^{\prime})(T) \right ) \quad \left ( (x,y)\in D_{T a} \right ).$$

\smallskip
\begin{proposition}\label{maxmonteta} The operator $\theta$ is maximal monotone and coercive.
\end{proposition}

\begin{proof}
To prove the maximal monotonicity of $\theta$, according to \cite[Proposition 2.2]{Br}, we have to show that: (i) $\theta$ is monotone and (ii) $\theta + i_d:D_{T a} \to \mathbb{R}^N \times \mathbb{R}^N$ is surjective (here, $i_d$ stands for the identity mapping on $\mathbb{R}^N \times \mathbb{R}^N$).

\medskip
\noindent (i) Let $(x,y), \, (\xi , \eta)\in D_{T a}$. The monotonicity inequality
$$ \langle \! \langle \theta (x,y) - \theta (\xi, \eta)| (x,y) - (\xi, \eta) \rangle \! \rangle \geq 0 $$
means
\begin{equation}\label{monteta}
\beta:=-\langle \phi(u_{x,y}^{\prime})(0)-\phi(u_{\xi , \eta}^{\prime})(0) | x-\xi \rangle+\langle \phi(u_{x,y}^{\prime})(T)-\phi(u_{\xi , \eta}^{\prime})(T) | y-\eta \rangle \geq 0.
\end{equation}
To see that this is true, we integrate by parts and make use of the boundary conditions ($u_{x,y}(0)=x, \, u_{x,y}(T)=y, \, u_{\xi , \eta}(0)=\xi, \, u_{\xi , \eta}(T)=\eta $) in
$$0=\int_0^T\langle \left [ \phi(u_{x,y}^{\prime}) -\phi( u_{\xi,\eta}^{\prime})\right]^{\prime}| \, u_{x,y} - u_{\xi , \eta} \rangle - \| u_{x,y}-u_{\xi, \eta}\|_{L^2}^2,$$
obtaining
$$0=\beta -\int_0^T\langle \phi(u_{x,y}^{\prime}) -\phi( u_{\xi,\eta}^{\prime}) | \, (u_{x,y}^{\prime} - u_{\xi , \eta}^{\prime} \rangle )- \| u_{x,y}-u_{\xi, \eta}\|_{L^2}^2$$
which, on account of the monotonicity of $\phi$, clearly implies \eqref{monteta}.

\medskip
\noindent (ii) For an arbitrary given and fixed $(\xi, \eta)\in {\mathbb{R}^N \times \mathbb{R}^N}$, we have to show that there is some $(x,y)= (x_{\xi}, y_{\eta})\in D_{T a}$ such that
\begin{equation}\label{surj}
\left( -\phi(u_{x,y}^{\prime})(0),\phi(u_{x,y}^{\prime})(T) \right )=\left(\xi-x,\eta - y \right ).
\end{equation}
In this view, we proceed by a fixed point argument, as follows. Let the operator $\Lambda = \Lambda _{\xi , \eta} : \mathbb{R}^N \times \mathbb{R}^N \to \mathbb{R}^N \times \mathbb{R}^N$ be defined by
$$\Lambda(x,y)= \left ( \overline{u}_{x-\xi, \eta-y} (0), \overline{u}_{x-\xi, \eta-y} (T) \right ) \quad \left ( (x,y) \in \mathbb{R}^N \times \mathbb{R}^N \right ).$$
Suppose that we have already proved that there exists a fixed point $(x,y)$ of $\Lambda$. This means
\begin{equation}\label{fixpoint}
\left ( \overline{u}_{x-\xi, \eta-y} (0), \overline{u}_{x-\xi, \eta-y} (T) \right )=(x,y),
\end{equation}
which reads: the solution $\overline{u}_{x-\xi, \eta-y}$ of the Neumann problem $({\mathcal N}_{x-\xi,\eta-y})$ satisfies the Dirichlet boundary conditions
\begin{equation}\label{contdir}
\overline{u}_{x-\xi, \eta-y}(0)=x, \quad \overline{u}_{x-\xi, \eta-y}(T)=y.
\end{equation}
Thus, $(x,y)\in D_{T a}$ by Proposition \ref{teu} $(i)$ and $\overline{u}_{x-\xi, \eta-y}=u_{x,y}$ which implies \eqref{surj}.

\smallskip
Therefore, it remains to prove that $\Lambda$ has a fixed point. We apply the a priori estimate method. By Proposition \ref{contubar} it is immediate that $\Lambda$  is continuous and hence compact on the finite dimensional space $\mathbb{R}^N \times \mathbb{R}^N$.
According to Schaefer's theorem (see e.g. \cite[Corollary 4.4.12]{Llo}), it suffices to show that there is a constant $c>0$ such that
\begin{equation}\label{constM}
|x|+|y| \leq c
\end{equation}
for all $(x,y) \in \mathbb{R}^N \times \mathbb{R}^N$ satisfying $(x,y)=\mu \, \Lambda (x,y)$ with some $\mu\in (0,1].$

\medskip
So, let $(x,y)\in \mathbb{R}^N \times \mathbb{R}^N$ be such that
\begin{equation*}\label{aest}
(x,y)= \mu \left (\overline{u}_{x-\xi, \eta-y}(0), \overline{u}_{x-\xi, \eta-y}(T)  \right )
\end{equation*}
with $\mu\in (0,1].$ From \eqref{ubar0} and \eqref{ubarT} we have
\begin{equation}\label{sys}
\left ( \frac{1}{\mu}+\frac{1}{T} \right ) x +\frac{1}{T} \, y= Q_1(x,y) \ \mbox{ and } \ \frac{1}{T} \, x + \left ( \frac{1}{\mu}+\frac{1}{T} \right )y= Q_2(x,y)
\end{equation}
where
$$ Q_1(x,y) :=\frac{1}{T} \left [ \xi + \eta + \int_0^Th  -\int_0^T \left [ \int_0^t \phi ^{-1} \left (x-\xi +\int_0^s ( \overline{u}_{x-\xi,\eta-y}-h)d \tau \right ) ds \right ] dt \right ],$$

$$Q_2(x,y):=\frac{1}{T} \left [ \xi + \eta +  \int_0^Th  -\int_0^T \left [ \int_t^T \phi ^{-1} \left (\eta -y +\int_T^s ( \overline{u}_{x-\xi,\eta-y}-h)d \tau \right ) ds \right ] dt \right ].$$
Solving \eqref{sys} we obtain
\begin{equation}\label{solsys}
x= \frac{ \displaystyle \left ( \frac{1}{\mu}+\frac{1}{T} \right ) Q_1(x,y)-\frac{1}{T}Q_2(x,y)}{\displaystyle \left ( \frac{1}{\mu}+\frac{1}{T} \right )^2-\left (\frac{1}{T} \right )^2}, \quad
y= \frac{ \displaystyle \left ( \frac{1}{\mu}+\frac{1}{T} \right ) Q_2(x,y)-\frac{1}{T}Q_1(x,y)}{\displaystyle \left ( \frac{1}{\mu}+\frac{1}{T} \right )^2-\left (\frac{1}{T} \right )^2}.
\end{equation}
Then, using that
$$\max \left \{ \left | Q_1(x,y) \right | \mbox{, } \left | Q_2(x,y) \right | \right \} \leq \frac{1}{T} \left [ \left |\xi \right | + \left | \eta \right | + \left |  \int_0^Th \right | +T^2 a \right ] =: \overline{Q}=\overline{Q}(\xi,\eta),$$
together with the assumption $\mu\in (0,1]$ and \eqref{solsys}, it is straightforward to see that \eqref{constM} holds true with $c=2\, \overline{Q}$ and the proof of the maximal monotonicity of $\theta$ is complete.

\medskip
We show that $\theta$ is coercive, i.e.,
\begin{equation}\label{coerciv}
\lim_{\scriptsize \begin{array}{crlc}
&\|(x,y)\| \to +\infty \\
&(x,y) \in D_{Ta}
\end{array}} \! \!
{\frac{\langle \! \langle \theta(x,y)|(x,y)\rangle \! \rangle }{\|(x,y)\|}=+\infty}.
\end{equation}
Let $(x,y)\in D_{Ta }$. Integrating by parts and using the boundary conditions ($u_{x,y}(0)=x$, $u_{x,y}(T)=y$) in
$$ -\int_0^T\langle \left [\phi(u_{x,y}^{\prime})\right]^{\prime} |\, u_{x,y}\rangle + \| u_{x,y}\|_{L^2}^2=\int_0^T\langle h | \, u_{x,y}\rangle  $$
we obtain
$$-\langle \! \langle \theta(x,y)|(x,y)\rangle \! \rangle+\int_0^T \langle \phi(u_{x,y}^{\prime}) |\, u_{x,y}^{\prime} \rangle+ \| u_{x,y}\|_{L^2}^2=\int_0^T\langle h | \, u_{x,y} \rangle,$$
which gives
\begin{eqnarray}\label{etunu}
\displaystyle \langle \! \langle \theta(x,y)|(x,y)\rangle \! \rangle & \geq & \| u_{x,y}\|_{L^2}^2 -\displaystyle  \int_0^T \langle h | \, u_{x,y} \rangle \nonumber \\
& \geq & \| u_{x,y}\|_{L^2} \left ( \| u_{x,y}\|_{L^2} - \sqrt {T} \|h\|_{\infty}\right ).
\end{eqnarray}

Let $t_0\in[0,T]$ be such that
$$\int_0^T|u_{x,y}(t)|dt=T |u_{x,y}(t_0)|.$$
From
$$|u_{x,y}(t_0)|\leq \frac{1}{\sqrt{T}}\|u_{x,y}\|_{L^2}$$
and
$$u_{x,y}(t)=u_{x,y}(t_0)+\int_{t_0}^tu_{x,y}^{\prime}(s)ds $$
it follows
\begin{equation}\label{ineqw}
|u_{x,y}(t)| \leq \frac{1}{\sqrt{T}}\|u_{x,y}\|_{L^2}+T a \qquad (t\in[0,T]).
\end{equation}
As $x=u_{x,y}(0)$, $y=u_{x,y}(T)$,  this yields
$$ \max \left \{ |x|, \, |y| \right \} \leq \frac{1}{\sqrt{T}}\|u_{x,y}\|_{L^2}+T a $$
and, hence
$$\| (x,y) \|= \sqrt{|x|^2+|y|^2}\leq \sqrt{2} \left (\frac{1}{\sqrt{T}}\|u_{x,y}\|_{L^2}+T a \right ),$$
or\begin{equation*}\label{etdoi}
\|u_{x,y}\|_{L^2} \geq \sqrt{T} \left ( \frac{1}{\sqrt{2}}\|(x,y)\|-Ta \right )
\end{equation*}
This together with \eqref{etunu} shows that, for $(x,y)\in D_{Ta}$ with $\|(x,y)\|>\sqrt{2} (Ta +\|h\|_{\infty})$, it holds
$$ \displaystyle \langle \! \langle \theta(x,y)|(x,y)\rangle \! \rangle \geq T\left ( \frac{1}{\sqrt{2}}\|(x,y)\|-Ta \right )\left (\frac{1}{\sqrt{2}}\|(x,y)\|-Ta-\|h\|_{\infty} \right ),$$
which clearly implies \eqref{coerciv} and the proof is complete.
\end{proof}

\medskip
\begin{remark}\label{remark22}
{\em ($i$) Observe that by the reasoning in the proof of inequality \eqref{ineqw} we can actually infer that any $v\in W^{1,\infty}$ with $\|v^{\prime}\|_{L^{\infty}} \leq a$
satisfies
\begin{equation}\label{ineqw1}
|v(t)| \leq \frac{1}{\sqrt{T}}\|v\|_{L^2}+T a \qquad (t\in[0,T]).
\end{equation}

\medskip
($ii$) It is worth comparing the proof of the maximal monotonicity of $\theta$ in our singular case with that in the classical case. More precisely, if for any $x,y\in \mathbb{R}^N$, we denote by $q_{x,y}$ the unique solution of the Dirichlet problem
$$-q''+q=h(t) \quad \mbox{ in } [0,T], \qquad q(0)=x, \, q(T)=y$$
and we set $\theta_0(x,y):=(-q_{x,y}'(0), q_{x,y}'(T))$, then it can be shown (see the proof of Proposition 3.1 from \cite{KaPamw}) that $\theta_0$ is monotone and continuous on $\mathbb{R}^N \times \mathbb{R}^N$. Thus, since $D(\theta_0)=\mathbb{R}^N \times \mathbb{R}^N$, it will follow from \cite[Proposition 2.4]{Br} that it is maximal monotone. By comparison, in our case - of the operator  $\theta$, it occurs the obstruction of the fact that $D(\theta)=D_{T a}\subsetneqq \mathbb{R}^N \times \mathbb{R}^N$ and hence the above argument can not be invoked. This is why to overcome this difficulty we have been led to provide a completely different proof.}
\end{remark}

\medskip
Now, we can state the main result of this section.

\smallskip
\begin{theorem}\label{rpmax}
If $\gamma $ is maximal monotone and $0_{\mathbb{R}^N \times \mathbb{R}^N}\in \gamma(0_{\mathbb{R}^N \times \mathbb{R}^N})$, then problem $[{\mathcal P}_\gamma (h)]$ has an unique solution $u_h$, for any $h\in C$.
\end{theorem}

\begin{proof} The uniqueness part is ensured by Proposition \ref{p1}. The prove the existence, let $\Gamma : \mathbb{R}^N \times \mathbb{R}^N \to 2^{\mathbb{R}^N \times \mathbb{R}^N}$ be defined by
$$D(\Gamma):=D(\gamma) \cap D_{T a}, \qquad \Gamma (x,y):= \gamma (x,y)+ \theta (x,y) \quad \left ( (x,y)\in D(\gamma ) \cap D_{T a} \right ).$$
Since $\theta$ is maximal monotone (Proposition \ref{maxmonteta}) and  $D(\gamma) \cap \mbox{int}D(\theta ) = D(\gamma) \cap D_{T a} \ni 0_{\mathbb{R}^N \times \mathbb{R}^N}$,  from \cite[Corollaire 2.7]{Br} we have that
$\Gamma$ is maximal monotone.

\medskip
We \textit{claim}  that $\Gamma : D(\Gamma) \to 2^{\mathbb{R}^N \times \mathbb{R}^N}$ is surjective.

\medskip
If $D(\Gamma)$ is bounded this follows by  \cite[Corollaire 2.2]{Br}. If this is not the case, we show that $\Gamma$ is coercive, in the sense that
\begin{equation}\label{coercivG}
\lim_{\scriptsize \begin{array}{crlc}
&\|(x,y)\| \to +\infty \\
&(x,y) \in D(\Gamma)
\end{array}} \! \!
{\frac{\displaystyle \inf \left \{ \langle \! \langle (\xi,\eta)|(x,y)\rangle \! \rangle\, : \, (\xi, \eta ) \in \Gamma (x,y) \right \} }{\|(x,y)\|}=+\infty}.
\end{equation}
According to \cite[Corollaire 2.4]{Br} (also, see \cite[Corollary 32.35]{Zeid}),  this will ensure the surjectivity of $\Gamma$, as claimed.

\medskip
So, let $(x,y)\in D(\Gamma)$ and $(\xi, \eta) \in \Gamma(x,y)$. To estimate $\langle \! \langle (\xi,\eta)|(x,y)\rangle \! \rangle$, let $(\zeta,\nu) \in \gamma (x,y)$ be with
$(\xi, \eta )=(\zeta, \nu)+\theta(x,y).$ Using that $\gamma$ is monotone and $0_{\mathbb{R}^N \times \mathbb{R}^N}\in \gamma(0_{\mathbb{R}^N \times \mathbb{R}^N})$, we obtain
$$\langle \! \langle (\xi,\eta)|(x,y)\rangle \! \rangle= \langle \! \langle (\zeta,\nu)-0_{\mathbb{R}^N \times \mathbb{R}^N}|(x,y)-0_{\mathbb{R}^N \times \mathbb{R}^N}\rangle \! \rangle + \langle \! \langle \theta(x,y)|(x,y)\rangle \! \rangle   \geq \langle \! \langle \theta(x,y)|(x,y)\rangle \! \rangle.   $$
This implies
$$\displaystyle \inf \left \{ \langle \! \langle (\xi,\eta)|(x,y)\rangle \! \rangle\, : \, (\xi, \eta ) \in \Gamma (x,y) \right \} \geq  \langle \! \langle \theta(x,y)|(x,y)\rangle \! \rangle$$
which, together with the coercivity of $\theta$ (see Proposition \ref{maxmonteta}), yields \eqref{coercivG}.

\medskip
Now, since $\Gamma$ is surjective, we have that there exists $(x,y)\in D(\gamma ) \cap D_{T a}$ such that
$$0_{\mathbb{R}^N \times \mathbb{R}^N}\in \gamma(x,y)+ \left( -\phi(u_{x,y}^{\prime})(0),\phi(u_{x,y}^{\prime})(T) \right ),$$
which means
$$\left( \phi(u_{x,y}^{\prime})(0),-\phi(u_{x,y}^{\prime})(T) \right )\in \gamma(u_{x,y}(0), u_{x,y}(T))$$

\smallskip
\noindent and the proof is accomplished by taking $u_h=u_{x,y}$.
\end{proof}

\medskip
\begin{remark}\label{remark3}{\em On account of Remark \ref{remark2} it is easily seen that in the particular case $N=1$ Theorem \ref{rpmax} holds true under the following hypothesis weaker than $(H_{\Phi })$: "\textit{$\phi:(-a,a) \to \mathbb{R}$ is an increasing homeomorphism}".}
\end{remark}

\section{A variational approach}\label{sectiunea3}
\noindent Under hypothesis $(H_{j})$, Theorem \ref{rpmax} will be employed with $\gamma= \partial j$. Thus, for $h\in C$, we consider the problem
$$
-\left[ \phi(u^{\prime}) \right] ^{\prime}+u = h(t),\qquad \left ( \phi \left( u^{\prime }\right)(0), -\phi \left( u^{\prime }\right)(T)\right )\in \partial j(u(0), u(T)).  \leqno{
%({\mathcal P}_\gamma)\equiv
[{\mathcal P}_{\partial j} (h)]
} $$
Since $\partial j$ is maximal monotone and $0_{\mathbb{R}^N \times \mathbb{R}^N}\in \partial j (0_{\mathbb{R}^N \times \mathbb{R}^N})$, from Theorem \ref{rpmax} this problem has an unique
solution $u_h$.

\medskip
Setting
$${\mathcal K}:= \left \{ v \in W^{1, \infty} \, : \, \|v^{\prime}\|_{_{L^{\infty}}}\leq a \right \},$$
we define $\Psi:C \to (-\infty , +\infty ]$ by
$$\Psi(v)=\displaystyle \left\{
\begin{array}{ll} \displaystyle
\int_0^T \left [\Phi(v^{\prime}) -\Phi (0) \right ] , & \hbox{if $v\in {\mathcal K}$;} \\
\\
+\infty, & \hbox{otherwise}.
\end{array}
\right.
$$
Then, following exactly the outline and the arguments from  the proof of Lemma 1 and \textit{Step I} in the proof of Proposition 1 from \cite{BeJeMaCV} (also, see \cite[Lemma 4.1]{BrMa}), we get that
the convex set ${\mathcal K}$ is closed and $\Psi$ is proper, convex and lower semicontinuous. We also introduce the function $J:C \to (-\infty , +\infty ]$ given by
$$ J(v)=j(v(0),v(T)) \qquad (v\in C).$$
Since $j$ is proper, convex and lower semicontinuous, it is immediate that the same hold true for $J$.
Notice that
\begin{equation}\label{dompj}
\begin{array}{rl}
D(\Psi + J)=D(\Psi) \cap D(J)  = & \left \{ v\in {\mathcal K} \, : \, (v(0),v(T))\in D(j)  \right \} \\
 = &\left \{ v\in {\mathcal K} \, : \, (v(0),v(T))\in D(j) \cap \overline{D}_{Ta} \right \},
\end{array}
\end{equation}
$u_h\in  D(\Psi + J)$ and $(u_h(0),u_h(T))\in {D}_{Ta}\cap D\left ( \partial j \right )$ (see \eqref{stripsol}). Also, as $\Psi$ is bounded from below and $J \geq 0$, it is clear that  $\Psi + J$ is bounded from below on $C$.

\begin{proposition}\label{minim}
The solution $u_h$ of problem $[{\mathcal P}_{\partial j} (h)]$ is the unique solution in $ D(\Psi + J)$ of the variational inequality
\begin{equation}\label{varinequality}
\displaystyle\int_0^T \left [\Phi(v^{\prime}) -\Phi (u^{\prime}) \right ]+J(v)-J(u)+\int_0^T \langle  u-h \, | \, v-u \rangle \geq 0, \quad \mbox{ for all }v \in D(\Psi + J)
\end{equation}
and the unique minimum point of $E:C \to (-\infty , +\infty ]$ defined by
\begin{equation*}
E(v)=\Psi(v) +J(v)+\frac{1}{2}\|v\|_{L^2}^2-\int_0^T\langle h \, | \, v \rangle \qquad (v\in C).
\end{equation*}
\end{proposition}

\begin{proof} First, using the elementary inequality
\begin{equation*}
\frac{|y|^2}{2}-\frac{|x|^2}{2} \geq \langle x\, | \, y-x \rangle  \qquad (x,y \in \mathbb{R}^N),
\end{equation*}
it is straightforward to see that any solution of \eqref{varinequality} is a minimum point of $E$. Therefore, it suffices to prove that $u_h$ solves \eqref{varinequality} and that it is
the unique minimum point of $E$ on $C$.

\medskip
Let $v\in {\mathcal K}$ be with $(v(0),v(T))\in D(j)$ and $v \neq u_h$ be arbitrarily chosen.
Using that, for all $(x,y) \in \mathbb{R}^N \times \mathbb{R}^N$ it holds
$$j(x,y)-j(u_h(0),u_h(T)) \geq \langle \phi(u_h^{\prime}(0)  | x-u_h(0) \rangle-\langle  \phi(u_h^{\prime}(T)  | y-u_h(T) \rangle,$$
we get
$$\begin{array}{ll} \displaystyle A&:= \displaystyle\int_0^T \left [\Phi(v^{\prime}) -\Phi (u_h^{\prime}) \right ]+J(v)-J(u_h)\\
 & \geq  \displaystyle \int_0^T \langle\phi(u_h^{\prime} )| v'-u_h^{\prime} \rangle+\langle \phi(u_h^{\prime}(0) ) | v(0)-u_h(0) \rangle-\langle \phi(u_h^{\prime}(T) ) | v(T)-u_h(T) \rangle.
 \end{array}
$$
On the other hand, as $u_h$ solves the equation in $[{\mathcal P}_{\partial j} (h)]$, we have
$$-\int _0^T \left \langle \left [ \phi(u_h^{\prime} ) \right ] ^{\prime}| v-u_h \right \rangle +\int _0^T\langle u_h |v-u_h \rangle=\int _0^T\langle h |v-u_h \rangle$$
and integration by parts formula gives
$$ \displaystyle \int_0^T \! \! \langle \phi(u_h^{\prime} )| v'-u_h^{\prime} \rangle +\langle \phi(u_h^{\prime}(0))  | v(0)-u_h(0) \rangle-\langle \phi(u_h^{\prime}(T))  | v(T)-u_h(T) \rangle$$ $$= \! \! \int _0^T \! \! \langle h-u_h |v-u_h \rangle$$
Thus, we obtain
\begin{equation}\label{estimA}
A \geq \int _0^T\langle h -u_h| \,v-u_h \rangle
\end{equation}
and this implies that $u_h$ is a solution of \eqref{varinequality}.
It remains to check that
\begin{equation*}
E(v) > E(u_h).
\end{equation*}
From
$$E(v)-E(u_h)=A+\frac{1}{2} \left ( \|v\|_{L^2}^2- \|u_h\|_{L^2}^2 \right )-\int _0^T\langle h |v-u_h \rangle
$$
and \eqref{estimA} it follows
$$E(v)-E(u_h) \geq \frac{1}{2}\|v-u_h\|_{L^2}^2 > 0$$
and the proof is complete.
\end{proof}

\medskip
Further, having in view hypothesis $(H_{F})$, we introduce ${\mathcal F}:C \to \mathbb{R}$ by setting
$$ {\mathcal F}(v):=-\int_0^TF(t,v)dt \qquad (v\in C).$$
 A standard reasoning (also see \cite[Remark 2.7]{PJ}) shows that ${\mathcal F}$ is of class $C^1$ on $C$ and its derivative is given by
$$\langle {\mathcal F}^{\prime} (u),v \rangle= -\int_0^T \langle \nabla_u F(t,u) |v \rangle dt \qquad (u,v\in C).$$

Now, \textit{the energy functional associated to problem} \eqref{sysp}-\eqref{bcsysp} will be the mapping ${\mathcal E}:C \to (-\infty , + \infty ]$ defined by
\begin{equation}\label{energy}
{\mathcal E}:=\Psi + J + {\mathcal F}.
\end{equation}
This fits the structure required by Szulkin's critical point theory \cite{Sz}. Accordingly, an element $u\in D(\Psi + J)$ is called \textit{a critical point of} ${\mathcal E}$ in \eqref{energy} if it satisfies
\begin{equation}\label{critpointdef}
(\Psi + J)(v)-(\Psi + J)(u) + \langle {\mathcal F}^{\prime} (u),v-u \rangle \geq 0, \quad \mbox{ for all }v \in D(\Psi + J).
\end{equation}
Also, a sequence $\{ u_n \} \subset D(\Psi + J)$ will be called a \textit{(PS)-sequence} if ${\mathcal E}(u_n) \to c \in \mathbb{R}$ and
\begin{equation}\label{PSseq}
\displaystyle\int_0^T \left [\Phi(v^{\prime}) -\Phi (u_n^{\prime}) \right ]+J(v)-J(u_n)+\langle {\mathcal F}^{\prime} (u_n),v-u_n \rangle
\end{equation}
\begin{equation*}
 \geq -\varepsilon_n\| v-u_n \|_{\infty}, \quad \mbox{ for all }v \in D(\Psi + J),
\end{equation*}

\medskip
\noindent where $\varepsilon_n \to 0$. The functional ${\mathcal E}$ is said to satisfy the \textit{(PS)-condition} if any (PS)-sequence has a convergent subsequence in $C$.

\medskip
\begin{theorem}\label{critpoint}
Any critical point of ${\mathcal E}$ is a solution of problem \eqref{sysp}-\eqref{bcsysp}.
\end{theorem}

\begin{proof}
Let $u\in D(\Psi + J)$ be a critical point of ${\mathcal E}$. By virtue of \eqref{critpointdef} it holds
$$\displaystyle\int_0^T \! \!  \left [\Phi(v^{\prime}) -\Phi (u^{\prime}) \right ]+J(v)-J(u)+\int_0^T \! \! \langle u-[u+ \nabla_u F(t,u)] \, |  v-u \rangle \geq 0, \quad \mbox{for all }v \in D(\Psi + J)$$
which is a variational inequality of type \eqref{varinequality} with $h=u+ \nabla_u F(t,u)$. From Proposition \ref{minim} we have that $u$ solves $[{\mathcal P}_{\partial j} (u+ \nabla_u F(t,u))]$, which actually means that it is a solution of \eqref{sysp}-\eqref{bcsysp}.
\end{proof}

\section{Minimum energy and saddle-point solutions}\label{sectiunea4}

\noindent We introduce the first eigenvalue-like constant
\begin{equation}\label{feig}
\lambda_1=\lambda_1(a,j):= \inf \left \{ \frac{\|v^{\prime} \|_{L^2}^2}{\|v \|_{L^2}^2} \, : \, v\in  D(\Psi + J), \; v \not\equiv 0_{\mathbb{R}^N}  \right \}.
\end{equation}

\begin{theorem}\label{l1poz} If $\lambda_1>0$ then the following are true:

\medskip
(i) for any $v\in  D(\Psi + J)$ one has
\begin{equation}\label{overl1}
|v(t)| \leq a \left ( \frac{1}{\sqrt{\lambda_1}} +T \right ) \qquad (t\in [0,T]) ;
\end{equation}

\medskip
(ii) ${\mathcal E}$ satisfies the (PS)-condition;

\medskip
(iii) ${\mathcal E}$ is bounded from below and attains its infimum at some $u\in D(\Psi + J)$, which is a critical
point of ${\mathcal E}$ and hence, a solution of problem \eqref{sysp}-\eqref{bcsysp}.
\end{theorem}

\begin{proof} $(i)$ Since $v\in {\mathcal K}$, it holds $\| v^{\prime} \|_{L^{2}} \leq a \sqrt{T}$ and by \eqref{feig} we have $\| v \|_{L^{2}} \leq a \sqrt{T/ \lambda_1}$.
Then \eqref{overl1} follows from \eqref{ineqw1}.

\medskip
$(ii)$ From \eqref{overl1} we infer that any $v\in  D(\Psi + J)$ satisfies
$$\|v\|_{W^{1, \infty }} \leq a \left ( \frac{1}{\sqrt{\lambda_1}} +T +1 \right ),$$
hence $D(\Psi + J)$ is bounded in $W^{1, \infty }$, and therefore relatively compact in $C$ -- by the compactness of the embedding $W^{1, \infty }\hookrightarrow C$. This clearly implies that
${\mathcal E}$ satisfies the (PS)-condition.

\medskip
$(iii)$ As $F$ is continuous on the compact set $[0,T] \times \overline{B}_{r}$, where $r=a \left ( {1}/{\sqrt{\lambda_1}} +T \right )$, from \eqref{overl1} we have that the mapping ${\mathcal F}$
is bounded on $D(\Psi + J)$. Thus, ${\mathcal E}$ is bounded from below on $C$. Then, using that ${\mathcal E}$ satisfies the (PS)-condition
and \cite[Theorem 1.7]{Sz}, there exists $u \in D(\Psi + J)$ such that ${\mathcal E}(u)= \inf _{C}{\mathcal E}.$ But, according to \cite[Proposition 1.1]{Sz} this means that $u$ is a critical point of ${\mathcal E}$,
and hence a solution of \eqref{sysp}-\eqref{bcsysp} by virtue of Theorem \ref{critpoint}.
\end{proof}

Having in view $(iii)$ in Theorem \ref{l1poz}, hypothesis $\lambda_1>0$ appears as being a sufficient one for the solvability  of  problem \eqref{sysp}-\eqref{bcsysp}. Thus, it becomes of interest to express this condition less technically. Doing this, we actually will provide an easier to be applied existence result.

\medskip
Let $\mbox {\sl cone}\, D(j)$ be the \textit{conical hull} of the convex set $D(j)$, that is
$${\sl cone}\, D(j):= \left \{ \alpha \, z \, : \, \alpha \geq 0, \, z\in D(j) \right \}.$$

\begin{corollary}\label{corl1} If   $\overline {\mbox {\sl cone}\, D(j)} \cap d^1_N = \left \{ 0_{\mathbb{R}^N \times \mathbb{R}^N} \right \}$ then problem \eqref{sysp}-\eqref{bcsysp} has at least one solution which is a minimizer of ${\mathcal E}$ on $C$.
\end{corollary}

\begin{proof} Let $M:= \left \{ v\in W^{1,2}([0,T]; \mathbb{R}^N) \, : \, (v(0),v(T)) \in \overline {\mbox {\sl cone}\, D(j)} \right \}$ and
$$ \underline{\lambda}_1:=\inf\left \{ \frac{\|v^{\prime} \|_{L^2}^2}{\|v \|_{L^2}^2}
\, : \, v\in  M , \; v \not\equiv 0_{\mathbb{R}^N}  \right \}.$$
From \cite[Theorem 3.1]{JePr1} we know that $ \underline{\lambda}_1>0$.
Then, since $D(\Psi + J) \subset M $, it holds $\lambda_1 \geq  \underline{\lambda}_1>0$ and the conclusion follows by Theorem \ref{l1poz} $(iii)$.
\end{proof}

\medskip
\begin{remark}\label{remark4}{\em $(i)$ Having in view \eqref{dompj}, one could suspect that if instead of ${\sl cone}\, D(j)$ one would consider ${\sl cone}\, \left [D(j) \cap \overline{D}_{Ta} \right ]$, then the result from Corollary \ref{corl1} would be sharper. But this is only apparent because it is straightforward to verify that, due to the fact that the convex set $D(j)$ contains $0_{\mathbb{R}^N \times \mathbb{R}^N}$, actually it holds ${\sl cone}\, D(j)={\sl cone}\, \left [ D(j) \cap \overline{D}_{Ta} \right ]$.

\medskip
$(ii)$ It is well known that the Dirichlet problem \eqref{sysp}-\eqref{dirh} is always solvable (see e.g. \cite[Corollary 6]{Ber-Maw0}). We observe that this "universal" existence result (meaning that no additional assumptions on the data in the equation \eqref{sysp} are made) is in our case immediately from  Corollary \ref{corl1}. Another immediate consequence of this is that the antiperiodic problem \eqref{sysp}-\eqref{aperi} has a solution. Or, slightly more general than \eqref{aperi}, if $K$ is the subspace
$$Y=\{ (x,y) \in \mathbb{R}^N \times \mathbb{R}^N \, : \, ax=by \} \quad (a,b \in \mathbb{R}, \,  |a|+|b|>0)$$
then
$$N_K(z)=Y^{\perp}=\{ (x,y) \in \mathbb{R}^N \times \mathbb{R}^N \, : \, bx=-ay \}$$
and \eqref{ncon} yields
\begin{equation}\label{nconape}
au(0)=bu(T), \qquad b\phi \left( u^{\prime }\right)(0)=a \phi \left( u^{\prime }\right)(T).
\end{equation}
We get that if $a \neq b$ then problem \eqref{sysp}-\eqref{nconape} is always solvable.  Thus, Corollary \ref{corl1}, and a fortiori Theorem \ref{l1poz} $(iii)$, highlights a whole class of boundary conditions for which such an universal existence result is valid. Another class for which this situation occurs and in which it is allowed $\lambda_1=0$ is provided in Corollary \ref{boundproj} below.
}
\end{remark}

For $v\in C$, we denote
$$\overline{v}:=\frac{1}{T}\int_0^Tv(t)dt, \quad \tilde{v}:= v-\overline{v}.$$
If $v\in W^{1, \infty}$, then each component $\tilde{v}_i$ of $\tilde{v}$  vanishes at some $\xi_i \in [0,T]$ and hence
$$|\tilde{v}_i(t)|=|\tilde{v}_i(t)-\tilde{v}_i(\xi)|\leq \int_0^T|\tilde{v}'_i(\tau)|d \tau \leq T\|v'\|_{L^{\infty}} \quad (i=\overline{1,N}, \, t\in[0,T]),$$
so, one has that
\begin{equation}\label{estvp}
\|\tilde{v}\|_{\infty} \leq T\sqrt{N}\|v'\|_{L^{\infty}}.
\end{equation}
For any $\delta >0$, set
$$D_{\delta}(\Psi + J):=\left \{v \in D(\Psi + J) \, : \, |\overline{v}| \leq \delta \right \}.$$

\begin{theorem}\label{minlem}
Assume that there is some $\delta>0$ such that
\begin{equation}\label{infeq}
\displaystyle \inf _{D_{\delta}(\Psi + J)}{\mathcal E}=\inf _{D(\Psi + J)}{\mathcal E}.
\end{equation}
Then ${\mathcal E}$ is bounded from below on $C$ and attains its infimum at some $u\in D_{\delta}(\Psi + J)$, which is a solution of problem \eqref{sysp}-\eqref{bcsysp}.
\end{theorem}
\begin{proof}
By virtue of \eqref{infeq} and $\displaystyle \inf _{D(\Psi + J)}{\mathcal E}=\inf _{C}{\mathcal E}$, it suffices to show that there is some
$u\in D_{\delta}(\Psi + J)$ such that
\begin{equation}\label{rinf}
\displaystyle{\mathcal E}(u)=\inf _{D_{\delta}(\Psi + J)}{\mathcal E}.
\end{equation}
This ensures that $u$ is a minimum point of ${\mathcal E}$ on $C$ and, on account of \cite[Proposition 1.1]{Sz}, it will be a critical point of ${\mathcal E}$. The proof will be accomplished by virtue
of Theorem \ref{critpoint}.

\medskip
From \eqref{estvp}, if $v\in D_{\delta}(\Psi + J)$, then
$$|v(t)| \leq |\overline{v}|+|\tilde{v}(t)|\leq \delta + aT\sqrt{N} \quad (t\in[0,T]).$$
This implies that $D_{\delta}(\Psi + J)$ is bounded in $W^{1, \infty}$ and ${\mathcal E}$ is bounded from below on $D_{\delta}(\Psi + J).$  By the compactness of the embedding $W^{1, \infty }\hookrightarrow C$ we infer that $D_{\delta}(\Psi + J)$ is relatively compact in $C$. Let $\{u_n\} \subset D_{\delta}(\Psi + J)$
be a minimizing sequence for ${\mathcal E}$. Passing to a subsequence if necessary, we may assume that $\{u_n\}$ converges uniformly to some $u\in {\mathcal K}$ (recall, ${\mathcal K}$ is a closed subset of $C$).  It is easy to check that $|\overline{u}| \leq \delta$.
Finally, using that ${\mathcal E}$ is lower semicontinuous, we get
$${\mathcal E}(u)\leq \liminf_{n\to \infty}{\mathcal E}(u_n)=\lim_{n\to \infty}{\mathcal E}(u_n)=\inf _{D_{\delta}(\Psi + J)}{\mathcal E} \quad (<+\infty)$$
which shows that necessarily $u\in D_{\delta}(\Psi + J)$ and \eqref{rinf} holds true.
\end{proof}

\medskip
For $A\subset \mathbb{R}^N \times \mathbb{R}^N$, let $\Pi_i A$ ($i=1,2$) be the projections of $A$, that is
$$\Pi_i A=\left \{ x_i \in \mathbb{R}^N \, : \, (x_1,x_2)\in A \mbox{ for some } x_k\in \mathbb{R}^N , \, i \neq k \right \}.$$

\begin{corollary}\label{boundproj}
If at least one of the sets $\Pi _1D(j)$ or $\Pi_2D(j)$ is bounded, then problem \eqref{sysp}-\eqref{bcsysp} has a solution.
\end{corollary}

\begin{proof} We assume that $\Pi _1D(j)$ is bounded; similar argument for the remaining case. Fix $c>0$ with $\Pi _1D(j)\subset \overline{B}_c.$ If $u\in D(\Psi + J)$, then from $(u(0), u(T))\in D(j)$ and
$\|u'\|_{L^{\infty}}\leq a$
it follows
$$|u(t)| \leq |u(0)|+\int_0^T|u'| d\tau \leq c+Ta,$$
which gives $|\overline{u}|\leq c+Ta$. We deduce that $D(\Psi + J)=D_{c+Ta}(\Psi + J)$ and Theorem \ref{minlem} applies with $\delta=c+Ta$.
\end{proof}

\medskip
\begin{corollary}\label{Fperiodic} Assume that:

\medskip (i) if $z\in D(j)$ then $z+\zeta \in D(j)$ and $j(z+\zeta)=j(z)$ for all $\zeta \in d_N^1$;

\medskip (ii) $F(t, u)$ is $\omega_i$-periodic ($\omega_i>0$) with respect to each $u_i$ ($i =\overline{1,N}$), for all $t\in[0,T]$.

\medskip
\noindent Then, for any given  $h\in C$ with $\overline{h}=0$, the system
\begin{equation}\label{sysper1}
-\left[ \phi(u^{\prime}) \right] ^{\prime}=\nabla_u F(t,u)+h(t) \quad \mbox{in }[0,T],
\end{equation}
has at least one solution satisfying the boundary condition \eqref{bcsysp}.
\end{corollary}

\begin{proof}
Setting
$$F_h(t,x):=F(t,x)+\langle h(t) |x \rangle \, \quad \left ( \, (t,x)\in [0,T] \times \mathbb{R}^N \, \right ),$$
it is easily seen that $F_h$ satisfies $(H_{F_h})$ -- that is $(H_{F})$ with $F_h$ instead of $F$, and thus attached with \eqref{sysper1}-\eqref{bcsysp}  will be the energy functional
${\mathcal E}_h:C \to (-\infty, +\infty]$ given by
$${\mathcal E}_h(u)=\Psi(u) + J(u) -\int_0^TF(t,u)dt- \int_0^T\langle h(t) |u \rangle dt \, \quad (u\in C).$$
From $(i)$ we deduce that if $u\in D(J)$ then, for any $y\in \mathbb{R}^N$, one has  $u+y \in D(J)$ and $J(u)=J(u+y)$.
Denote by $e_1, \dots, e_N$ the canonical basis in $\mathbb{R}^N$ and set $\omega:=\max \{ \omega_i \, : \, i=\overline{1,N} \}$. Let $u\in D( \Psi + J)$ and $k_i \in \mathbb{Z}$ be such that
$\langle \overline{u} | e_i \rangle -k_i \, \omega_i \in [0, \omega_i)$ ($i=\overline{1,N}$); such a family of $k_i$ exists and is unique.
Using the periodicity of $F$  and $\overline{h}=0$ we get
\begin{equation}\label{perErond}
{\mathcal E}_h(u)={\mathcal E}_h \left (u-\sum_{i=1}^{N}k_i \, \omega_i \, e_i \right ).
\end{equation}
Put
$$ \overline{u}^0:= \sum_{i=1}^N \left ( \langle \overline{u} | e_i \rangle -k_i \, \omega_i \right )e_i$$
and define $v:= \overline{u}^0+\tilde{u}$. As $\overline{v}=\overline{u}^0$ and $\tilde{v}=\tilde{u}$, from \eqref{perErond} it follows
$${\mathcal E}_h(u)={\mathcal E}_h \left (\tilde{u}+\overline{u}-\sum_{i=1}^{N}k_i \, \omega_i \, e_i \right )={\mathcal E}_h \left (\tilde{v}+\overline{v} \right )={\mathcal E}_h(v).$$
Since $|\overline{v}|\leq \omega \sqrt{N}$, we obtain
$$\left \{{\mathcal E}_h(u)\, : \, u\in D( \Psi + J) \right \}=\left \{{\mathcal E}_h(v)\, : \, v\in D_{\omega \sqrt{N}}( \Psi + J) \right \}$$
and Theorem \ref{minlem} yields the conclusion.
\end{proof}

\medskip
We introduce the set
$$S_{\sigma}:=\left\{ \begin{array}{clrc}
d_N^1 & \hbox{if $\sigma=0$}, \\
\overline{D}_{\sigma} & \hbox{if $0<\sigma <\infty$},\\
\mathbb{R}^N \times \mathbb{R}^N  & \hbox{if $\sigma =\infty$}
\end{array}
\right.$$
and analyze the boundary condition \eqref{ncon} with $K=S_{\sigma}$ ($0\leq \sigma \leq \infty$), under the additional assumption on $g=g(x,y)$ that $\nabla g= \left ( \nabla_x g  \, , \, \nabla_y g \right )$ exists and  is continuous.
Then \eqref{ncon} becomes

\begin{equation}\label{nconstrip}
\left\{ \begin{array}{clrc}
& \! \! \! \! \! \left ( u(0), u(T) \right ) \in S_{\sigma}, \\
&\vspace{0.0cm}\\
& \! \! \! \! \!  \left ( \phi \left( u^{\prime }\right)(0)-\nabla _x g ( u(0), u(T)), -\phi \left( u^{\prime }\right)(T)-\nabla _y g ( u(0), u(T))\right ) \in N_{S_{\sigma}}(u(0), u(T))
\end{array}
\right.
\end{equation}
and, recall it is nothing else but condition \eqref{bcsysp} with
\begin{equation}\label{partjS}
j=g+I_{S_{\sigma}}.
\end{equation}

\medskip
\begin{proposition}\label{condsigma} Assume that $g=g(x,y):\mathbb{R}^N \times \mathbb{R}^N \to \mathbb{R}$ is convex, $\nabla g= \left ( \nabla_x g  \, , \, \nabla_y g \right )$ exists and is continuous,
$\nabla_x g(0_{\mathbb{R}^N \times \mathbb{R}^N}) = 0_{\mathbb{R}^N}=\nabla_y g(0_{\mathbb{R}^N \times \mathbb{R}^N})$
and  $g(0_{\mathbb{R}^N \times \mathbb{R}^N})=0$. Let $u$ be a solution of system \eqref{sysp} $($or \eqref{sysper1} with $h\in C$$)$ subject to the boundary condition \eqref{nconstrip}. Then the following are true:

\medskip (i) if $\sigma=0$ then $u$ satisfies
\begin{equation}\label{persteklov}
\left\{ \begin{array}{clrc}
& \! \! \! \! \! u(0)= u(T), \\
& \! \! \! \! \!  \phi \left( u^{\prime }\right)(0)-\nabla _x g ( u(0), u(T))=\phi \left( u^{\prime }\right)(T)+\nabla _y g ( u(0), u(T));
\end{array}
\right.
\end{equation}

\medskip (ii) if $T a \leq \sigma \leq \infty$ then $u$ satisfies the Neumann-Steklov type boundary condition
\begin{equation}\label{neumsteklov}
\phi \left( u^{\prime }\right)(0)=\nabla _x g ( u(0), u(T)), \quad \phi \left( u^{\prime }\right)(T)=-\nabla _y g ( u(0), u(T));
\end{equation}

\medskip (iii) if $0<\sigma<Ta $ then $u$ satisfies
\begin{equation}\label{bcondsigma}
\left\{ \begin{array}{clrc}
\mbox{ either } |u(0)-u(T)|< \sigma  \mbox{ and }  \\
 \phi \left( u^{\prime }\right)(0)=\nabla _x g ( u(0), u(T)), \quad \phi \left( u^{\prime }\right)(T)=-\nabla _y g ( u(0), u(T)) \\
\mbox{ or }\\
|u(0)-u(T)| = \sigma  \mbox{ and } \\
\phi \left( u^{\prime }\right)(0)-\nabla _x g ( u(0), u(T)) = s (u(0)-u(T))=\phi \left( u^{\prime }\right)(T)+\nabla _y g ( u(0), u(T)) \\
 \mbox{ for some }s\geq 0.
\end{array}
\right.
\end{equation}

\end{proposition}

\begin{proof}
We discuss all the possible cases with respect to  $\sigma \in [0,\infty ]$.

\medskip
\noindent $\underline{\mbox{\textit{Case} }\sigma=0}$:  $S_0=d_N^1$, $N_{S_0}(z)=d_N^2$ if $z\in d_N^1$ and  \eqref{nconstrip} becomes  \eqref{persteklov}.

\medskip
\noindent $\underline{\mbox{\textit{Case} }\sigma= \infty}$:  $S_{\infty}=\mathbb{R}^N \times \mathbb{R}^N$, $N_{S_{\infty}}(z)=\left \{ 0_{\mathbb{R}^N \times \mathbb{R}^N} \right \}$ for all $z\in \mathbb{R}^N \times \mathbb{R}^N$
and \eqref{nconstrip} gives  \eqref{neumsteklov}.

\medskip
\noindent $\underline{\mbox{\textit{Case} }0 < \sigma<\infty}$: $S_{\sigma}=\overline{D}_{\sigma}=\{z=(z_1,z_2) \, : \, |z_1 -z_2| \leq \sigma \}$,
\begin{equation}\label{connormal}
N_{S_{\sigma}}(z)=N_{\overline{D}_{\sigma}}(z)=\left\{ \begin{array}{clrc}
\{ 0_{\mathbb{R}^N \times \mathbb{R}^N} \} & \hbox{if $|z_1-z_2|< \sigma$}, \\
\left \{ s (z_1-z_2, z_2-z_1) \, : \, s\geq 0 \right \} & \hbox{if $|z_1-z_2|=\sigma $},
\end{array}
\right. \, (z\in S_{\sigma}).
\end{equation}
For the convenience of the reader, we have included in the Appendix section a proof of \eqref{connormal} when $|z_1-z_2|=\sigma $.

\medskip
\noindent $\underline{\mbox{\textit{Subcase} }Ta \leq \sigma<\infty}$: As $|u(0)-u(T)|<Ta\leq \sigma$, from \eqref{connormal} we have $N_{S_{\sigma}}(u(0),u(T))=\{ 0_{\mathbb{R}^N \times \mathbb{R}^N} \}.$
It follows that $u$ satisfies  \eqref{neumsteklov}.

\medskip
\noindent $\underline{\mbox{\textit{Subcase} }0 < \sigma<Ta}$: If $|u(0)-u(T)|< \sigma$, then $u$ satisfies  \eqref{neumsteklov} -- as above,
while if $|u(0)-u(T)|= \sigma$
from \eqref{connormal} it holds
$$N_{S_{\sigma}}(u(0),u(T))=\left \{ s (u(0)-u(T), u(T)-u(0)) \, : \, s\geq 0 \right \}$$
and \eqref{nconstrip} yields
$$ \phi \left( u^{\prime }\right)(0)-\nabla _x g ( u(0), u(T)) = s (u(0)-u(T))=\phi \left( u^{\prime }\right)(T)+\nabla _y g ( u(0), u(T)) \mbox{ for some }s\geq 0.$$
Therefore, in this subcase $u$ satisfies \eqref{bcondsigma}.

\end{proof}

\medskip
\begin{example}\label{primex}{\em
Let $f:\mathbb{R}^N \to \mathbb{R}$ be of class $C^1$, convex, with $\nabla f (0_\mathbb{R}^N)=0_{\mathbb{R}^N}$ and $f(0_{\mathbb{R}^N})=0.$
Setting $g(x,y):=f(x-y)$ ($x,y \in \mathbb{R}^N$), we have that $g$ is convex, of class $C^1$ on $\mathbb{R}^N \times \mathbb{R}^N$,
$$\nabla g(x,y)= \left ( \nabla f(x-y), -\nabla f(x-y) \right ),$$
$g|_{d_N^1}\equiv 0$ and $\nabla g|_{d_N^1}\equiv0_{\mathbb{R}^N \times \mathbb{R}^N}.$

\medskip
Observe that with this $g$, function  $j$ in \eqref{partjS} satisfies condition $(i)$ from Corollary \ref{Fperiodic}. Thus, under the assumptions on  $F$ and $h$ in Corollary \ref{Fperiodic}, using Proposition \ref{condsigma}
we deduce that system \eqref{sysper1} has at least one solution satisfying one of the boundary conditions: the periodic \eqref{peri},
\begin{equation}\label{neumfxy}
\phi \left( u^{\prime }\right)(0)=\nabla f ( u(0)- u(T))=\phi \left( u^{\prime }\right)(T)
\end{equation}
or, for any $0<\sigma <Ta$,
\begin{equation}\label{bcondsigmafxy}
\left\{ \begin{array}{clrc}
\mbox{ either } |u(0)-u(T)|< \sigma  \mbox{ and }  \\
 \phi \left( u^{\prime }\right)(0)=\nabla f ( u(0)- u(T))=\phi \left( u^{\prime }\right)(T) \\
\mbox{ or }\\
|u(0)-u(T)| = \sigma  \mbox{ and } \\
\phi \left( u^{\prime }\right)(0)=\nabla  f ( u(0)- u(T)) + s (u(0)-u(T))=\phi \left( u^{\prime }\right)(T) \\
 \mbox{ for some }s\geq 0.
\end{array}
\right.
\end{equation}
For instance, if $f(x)=(e^{|x|^2}-1)/2$ ($x\in \mathbb{R}^N$), then \eqref{neumfxy} becomes
\begin{equation*}
\phi \left( u^{\prime }\right)(0)=\displaystyle e^{|u(0)-u(T)|^2} ( u(0)- u(T))=\phi \left( u^{\prime }\right)(T)
\end{equation*}
and in a similar manner can be written \eqref{bcondsigmafxy}.

\medskip
In the light of the above, we can assert that the significance of Corollary \ref{Fperiodic} occurs twofold toward the class of problems with periodic potential $F$.
On a hand, it covers known results, such as the solvability of periodic and Neumann problems  (see \cite[Theorem 2]{BrMa0}, \cite[Theorem 9.1]{BrMa} and Theorems 4.2 and 4.4 in \cite{JeSe1}) through a unitary approach and on the other hand it points out a wider class of boundary conditions for which such a result still holds true.
}
\end{example}

\medskip

\begin{corollary}\label{talp1}
If $\nabla_u F$ is bounded on  $[0,T] \times \mathbb{R}^N$ and
\begin{equation}\label{ALP1}
\lim_{|x|\to\infty}\int_0^TF(t,x)dt=-\infty,
\end{equation}
then problem \eqref{sysp}-\eqref{bcsysp} has at least one solution.
\end{corollary}

\begin{proof} Let $c_1>0$ be such that $|\nabla_uF|\leq c_1$. For $u\in {\mathcal K}$, using \eqref{estvp}, we obtain
\begin{equation}
\begin{array}{ll}\displaystyle \int_0^TF(t,u)dt&= \displaystyle \int_0^T\int_0^1\frac{d}{ds}F(t,\overline{u}+s\tilde{u}(t) )ds \, dt +\int_0^TF(t,\overline{u})dt \\
&= \displaystyle  \int_0^T\int_0^1\langle \nabla_u F(t, \overline{u}+s\tilde{u}(t)|\tilde{u}(t) \rangle ds \, dt +\int_0^TF(t,\overline{u})dt \label{incercare} \\
&\leq \displaystyle c_1T^2a\sqrt{N}+\int_0^TF(t,\overline{u})dt.
\end{array}
\end{equation}
Then, as $\Psi+J$ is bounded from below, say by $-c_2<0$, we have the estimate
$$\begin{array}{ll}\displaystyle {\mathcal E}(u)&=(\Psi + J)(u) + {\mathcal F}(u) \displaystyle \geq -c_2-\int_0^TF(t,u)dt\\
& \geq \displaystyle -c_2 -c_1T^2a\sqrt{N}-\int_0^TF(t,\overline{u})dt \quad (u\in D(\Psi + J)).
\end{array}$$
By virtue of \eqref{ALP1} there exists $\delta >0$ such that ${\mathcal E}(u) >0$ provided that $u\in D(\Psi + J)$ and $|\overline{u}|> \delta$. Since ${\mathcal E}(u)(0)=0$, we have that \eqref{infeq} is fulfilled and the conclusion follows from Theorem \ref{minlem}.
\end{proof}

\medskip
\begin{example}\label{secex}{\em (See p. 13-14 in \cite{MaWi} for the periodic classical scalar variant.)
Let $\rho >0$, $0<\beta <\pi$, $h\in C$ with $\overline{h}=0$ and consider the system
\begin{equation}\label{exsysALP}
-\left[ \phi(u^{\prime}) \right] ^{\prime}=\rho \left [ \sin (\beta-|u|) - \sin \beta \right ] \frac{u}{|u|} +h(t) \quad \mbox{in }[0,T].
\end{equation}
In this case,
$$F(t,u)=\rho \left [   \cos \left ( |u| - \beta \right ) - \cos \beta -\left ( \sin \beta \right ) |u| \right ] + \langle h(t) \, | \, u \rangle$$
and hence
$$\int_0^TF(t,x)dt=T\rho \left [   \cos \left ( |x| - \beta \right ) - \cos \beta -\left ( \sin \beta \right ) |x| \right ] \rightarrow -\infty \quad \mbox{ as }|x| \to \infty.$$
In addition, $\nabla _uF$ is bounded by $2 \rho + \|h \|_{\infty}$ and problem \eqref{exsysALP}-\eqref{bcsysp} has at least one solution by Corollary \ref{talp1}.

\medskip In the scalar case ($N=1$), for $\beta = 0$ and $\rho$, $h$ as above, one has the singular pendulum equation
$$\left[ \phi(u^{\prime}) \right] ^{\prime}=\rho  \sin u -h(t) \quad \mbox{in }[0,T]$$
and its solvability is covered by Corollary \ref{Fperiodic}.
}
\end{example}

\medskip
It is a natural question asking what happens when instead of \eqref{ALP1} the complementary condition occurs -- that is, the limit is $+\infty$ instead of $-\infty$. We will answer this in the next theorem, with the remark that the solution we will obtain will no longer be a minimum point for the energy ${\mathcal E}$, but one of saddle-point type. In addition, we will see that in that case additional conditions on $j$ are necessary, unlike the situation in Corollary \ref{talp1} when any condition of the type \eqref{bcsysp} is allowed.

\begin{proposition}\label{convK}
If $\{u_n\}\subset {\mathcal K}$ is a sequence such that $\{\overline{u}_n\}$ is bounded in $\mathbb{R}^N$, then it possesses a convergent subsequence.
\end{proposition}

\begin{proof} From \eqref{estvp}  we have
$$\|u_n\|_{W^{1,\infty}} \leq|\overline{u}_n|+ (T \sqrt{N}+1)a$$
showing that $\{u_n\}$ is bounded in $W^{1, \infty}$. The conclusion follows by the compactness of the embedding $W^{1, \infty }\hookrightarrow C$.
\end{proof}

\begin{theorem}\label{talp2}
Assume that $j|_{d_N^1}\equiv 0$ and $j|_{D(j)}$ is bounded. If $\nabla_u F$ is bounded on  $[0,T] \times \mathbb{R}^N$ and
\begin{equation}\label{ALP2}
\lim_{|x|\to\infty}\int_0^TF(t,x)dt=+\infty,
\end{equation}
then problem \eqref{sysp}-\eqref{bcsysp} has at least one solution.
\end{theorem}

\begin{proof} Using Saddle Point Theorem \cite[Theorem 3.5]{Sz} we show that ${\mathcal E}$ has a critical point and then the conclusion will follow by Theorem \ref{critpoint}. With this aim, let us split $C=\mathbb{R}^N\oplus\widetilde{C}$, where $\widetilde{C}=\{u\in C \, : \, \overline{u}=0_{\mathbb{R}^N} \}.$ Since for $u=x\in \mathbb{R}^N$, one has $J(x)=0$, from \eqref{ALP2} it follows
\begin{equation}\label{spt1}
{\mathcal E}(x)=-\int_0^TF(t,x)\, dt \to -\infty , \quad \mbox{ as } |x| \to \infty.
\end{equation}
Since $\nabla_u F$ is bounded, proceeding exactly as in the proof of Corollary
\ref{talp1}, we obtain the existence of a constant $k_1\in \mathbb{R}$ with
\begin{equation*}
{\mathcal E}(u) \geq k_1-\int_0^TF(t, \overline{u}) \, dt \qquad  ( u \in D(\Psi+J)),
\end{equation*}
showing that
\begin{equation*}
{\mathcal E}(u)\geq k_1-\int_0^TF(t, 0_{\mathbb{R}^N}) \, dt =:k_2 \in \mathbb{R}\qquad ( u \in D(\Psi+J) \cap\widetilde{C}).
\end{equation*}
From \eqref{spt1} there are constants $\rho >0$ and $\alpha_1 < k_2$ so that
\begin{equation}\label{gspt}
{\mathcal E}|_{\{ x\in \mathbb{R}^N \, :  \, |x|=\rho \}} \leq \alpha _1 \quad \mbox{and} \quad {\mathcal E}|_{\widetilde{C}} \geq k_2.
\end{equation}
\medskip
It remains to check that ${\mathcal E}$ satisfies the (PS)-condition. Let $\{ u_n \} \subset D(\Psi + J)$ be a (PS)-sequence. First, observe that again by the boundedness of $\nabla_u F$  we can infer (see \eqref{incercare}) that the sequence
$$\left \{ \int_0^T \left [ F(t,u_n) - F(t, \overline{u}_n)\right ] \, dt \right \}$$
is bounded. Next, since $\{ {\mathcal E}(u_n) \}$ and $\{ J(u_n) \}$ are also bounded, from the writing
$${\mathcal E}(u_n)=\int_0^T \left [\Phi(u_n^{\prime}) -\Phi (0) \right ] +J(u_n)-\int_0^T  F(t, \overline{u}_n)\, dt- \int_0^T \left [ F(t,u_n) - F(t, \overline{u}_n)\right ] \, dt $$
it follows that there is a constant $k_3 \in \mathbb{R}$ such that
\begin{equation}\label{marginint}
\int_0^T  F(t, \overline{u}_n)\, dt \leq k_3 \qquad (n\in \mathbb{N}).
\end{equation}
Then, by \eqref{ALP2} the sequence $\{ \overline{u}_n \}$  is bounded and Proposition \ref{convK} ensures that $\{ u_n \}$ has a convergent subsequence in $C$. Consequently, ${\mathcal E}$ satisfies the (PS)-condition and
the proof is complete.

\end{proof}

\medskip
\begin{example}\label{triex}{\em
With $f$ and $g$ as in Example \ref{primex}, function $j$ given by \eqref{partjS} -- with $0 \leq \sigma < \infty$, satisfies the requirements of Theorem \ref{talp2}.
We get that if  $\rho$, $h$ are as in Example \ref{secex} and $-\pi < \beta <0$, then system \eqref{exsysALP}  subject to one of the boundary conditions
\eqref{peri}, \eqref{neumfxy} or \eqref{bcondsigmafxy} -- with any $0 < \sigma <Ta$, has at least one solution.
}
\end{example}

\medskip
\begin{remark}\label{remarkALP}{\em Theorem \ref{talp2} and Corollary \ref{talp1} are results of Ahmad-Lazer-Paul type \cite{ALP} and are known for the periodic and Neumann boundary conditions \cite{JeSe1}. It is also worth noting that Theorem \ref{talp2} does more than answer an open problem regarding existence of saddle-point solutions for the periodic problem raised in \cite[Remark 7.4]{BrMa}. Namely, it highlights through $j$ a whole class of boundary problems which under  semi-coerciveness condition  \eqref{ALP2} admit solutions that appear as  saddle-points of the energy ${\mathcal E}.$
}
\end{remark}

If in Theorem \ref{talp2} we require coerciveness instead of semi-coerciveness of $F$, then we can give up the bounding of $\nabla _uF$. In this respect, the following theorem removes the hypothesis of boundedness of $\nabla _uF$ from Corollary \ref{talp1} and Theorem \ref{talp2}.

\begin{theorem}\label{talpcoerc} If either
\begin{equation}\label{lessc}
\limsup_{|x| \to \infty}F(t,x)<\min \Phi-\Phi(0) \quad \mbox{ uniformly with }t\in [0,T],
\end{equation}
or $j|_{d_N^1}\equiv 0$,  $j|_{D(j)}$ is bounded and
\begin{equation}\label{ALP2coer}
\lim_{|x|\to\infty}F(t,x)dt=+\infty \quad \mbox{ uniformly with }t\in [0,T],
\end{equation}
then problem \eqref{sysp}-\eqref{bcsysp} has at least one solution.
\end{theorem}

\begin{proof} If \eqref{lessc} holds true, then there are numbers $\varepsilon >0 <\rho_{\varepsilon}$ so that if $|x| > \rho_{\varepsilon}$, then
$$ F(t,x)<\min \Phi-\Phi(0)-\varepsilon \quad \mbox{ for all }t\in [0,T].$$
If $u\in D(\Psi)$ is with $|\overline{u}|>\rho_{\varepsilon}+Ta \sqrt{N}$, from \eqref{estvp} it follows $|u(t)|> \rho_{\varepsilon}$ and, hence
$$F(t,u(t))<\min \Phi-\Phi(0)-\varepsilon \quad \mbox{ for all }t\in [0,T].$$
Then, using the estimation (recall, $J(u) \geq 0$):
$${\mathcal E}(u) \geq T \left (\min \Phi-\Phi(0) \right ) -\int_0^T F(t,u(t))dt>T \varepsilon$$
and ${\mathcal E}(0)=0$, we get the conclusion by Theorem \ref{minlem} which applies with $\delta=\rho_{\varepsilon}+Ta \sqrt{N}.$

\smallskip
For the remaining case, observe that \eqref{ALP2coer} implies \eqref{ALP2} and similarly to the proof of Theorem \ref{talp2}, we can show that ${\mathcal E}$ has the geometry required by Saddle Point Theorem (see \eqref {gspt}).
To see that ${\mathcal E}$ satisfies (PS)-condition, let $\{ u_n \} \subset D(\Psi + J)$ be a (PS)-sequence. We claim that $\{ |\overline{u}_n| \}$ is bounded. Indeed, otherwise, we may assume -- going if necessary to a subsequence, that $|\overline{u}_n| \to \infty$, as $n \to \infty$. Using \eqref{estvp} and \eqref{ALP2coer}, we deduce that
$$F(t,u_n(t)) \to +\infty \quad \mbox{ uniformly with }t\in [0,T], \mbox{ as }n \to \infty.$$
This implies
$$\int_0^TF(t,u_n(t))dt \to +\infty \quad \mbox{ as }n \to \infty $$
and, as ${\mathcal E}(u_n)$, $\Psi(u_n)$, $J(u_n)$ are bounded, from the writing
$${\mathcal E}(u_n)=\Psi(u_n)+J(u_n)-\int_0^TF(t,u_n(t))dt$$
we get a contradiction. Thus, $\{ \overline{u}_n \}$  is bounded and by Proposition \ref{convK}, $\{ u_n \}$ has a convergent subsequence in $C$. Therefore, ${\mathcal E}$ satisfies the (PS)-condition.
Now, the conclusion follows by Saddle Point Theorem and Theorem \ref{minlem}.
\end{proof}

\section{Appendix}

We claim that if $z=(z_1,z_2)\in \mathbb{R}^N \times \mathbb{R}^N$ is such that $|z_1-z_2|=\sigma$, then
$$N_{{\overline D}_{\sigma}}(z)=\left \{ s (z_1-z_2, z_2-z_1) \, : \, s\geq 0 \right \}. \leqno{(A.1)}$$
Indeed, first observe that if $x=(x_1,x_2)\in {\overline D}_{\sigma}$, then
$$\langle \!  \langle  (z_1-z_2, z_2-z_1) | x-z \rangle \! \rangle =\langle z_1-z_2 | x_1-z_1 \rangle+ \langle z_2-z_1 | x_2-z_2 \rangle \leq 0,$$
means
$$\langle z_1-z_2 | x_1-z_1 - (x_2-z_2)\rangle \leq 0,$$
or
$$\langle z_1-z_2 |x_1-x_2 \rangle \leq \langle z_1-z_2 |z_1-z_2 \rangle=\sigma ^2,$$
which clearly is true, because
$$\langle z_1-z_2 |x_1-x_2 \rangle \leq |z_1-z_2| \, |x_1-x_2|\leq \sigma ^2.$$
Thus, we have
$$\langle \!  \langle s (z_1-z_2, z_2-z_1) | x-z \rangle \! \rangle\leq 0 \quad \mbox{ for all }x \in {\overline D}_{\sigma}\mbox{ and }s \geq 0,$$
meaning that $N_{{\overline D}_{\sigma}}(z) \supset \left \{ s (z_1-z_2, z_2-z_1) \, : \, s\geq 0 \right \}$ and it remains to show the reverse inclusion. In this respect, since if $(\xi_1,\xi_2)\in N_{{\overline D}_{\sigma}}(z)$, then
$$\langle \xi_1 | x_1-z_1 \rangle+ \langle \xi_2 | x_2-z_2 \rangle=\langle \!  \langle  (\xi_1, \xi_2) | x-z \rangle \! \rangle \leq 0 \quad \mbox{ for all } x=(x_1,x_2)\in {\overline D}_{\sigma}$$
and taking $x_j=z_j+y$ ($j=1,2$), it holds $\langle \xi_1 + \xi_2| y \rangle\leq 0$ for any $y\in \mathbb{R}^N$, we obtain that necessarily $(\xi_1,\xi_2)=(\xi, -\xi)$ with some $\xi \in \mathbb{R}^N$ satisfying
$$\langle \xi | x_1-z_1 - x_2+z_2 \rangle \leq 0 \quad \mbox{ for all } x=(x_1,x_2)\in {\overline D}_{\sigma}. \leqno{(A.2)}$$
If $\xi \neq 0_{\mathbb{R}^N}$, taking $x_1=\xi |z_1-z_2| / |\xi|$ and $x_2=0_{\mathbb{R}^N}$ in $(A.2)$, we get
$$|\xi| \, |z_1-z_2|\leq \langle \xi | z_1-z_2 \rangle,$$
which implies that equality holds in Cauchy-Schwartz inequality (i.e. $|\langle \xi | z_1-z_2 \rangle|=|\xi| \, |z_1-z_2|$).
But this means that $\xi=\tau (z_1-z_2)$ with some $\tau \neq 0$ and, using again $(A.2)$ with $x_1=x_2=0_{\mathbb{R}^N}$, we infer that $\tau >0$. The proof of $(A.1)$ is accomplished.

\vspace{0.5cm}
\noindent {\bf Acknowledgement}. I am grateful to Radu Precup for useful discussions which improved the presentation of the paper.

\end{document}